\documentclass[a4paper,twoside,11pt]{article}
\usepackage[english]{babel}
\usepackage[latin1]{inputenc}
\usepackage{indentfirst}
\usepackage{amssymb}
\usepackage{amsmath}
\usepackage{amsthm}
\usepackage{amsmath,amsfonts,amsthm,bm}

\Large
\theoremstyle{plain}
\newtheoremstyle{miestilo}{12pt}{\topsep}{\itshape}{}{\bf}{}{ }{}
\swapnumbers
\theoremstyle{miestilo}
\newtheorem{teorema}[subsection]{Theorem.}
\newtheorem{proposicion}[subsection]{Proposition.}
\newtheorem{lema}[subsection]{Lemma.}

\newtheoremstyle{misnotas}{12pt}{12pt}{}{}{\bf}{}{ }{\remark}
\theoremstyle{misnotas}
\newtheorem{remark}[subsection]{\ {\bf Remark.}}

\newtheorem{apartado}[subsection]{\ {\ }}
\allowdisplaybreaks[2]
\begin{document}
\flushbottom

\def \C{\mathcal C}
\def \E{\mathcal E}
\def \Hom{{\rm Hom}}
\def \End{{\rm End}}
\def \boPHI{\mathbf{\Phi}}

%
%

{\Large  {\centerline {\bf On von Neumann regularity}}}
{\Large   {\centerline  {\bf of algebras of quotients of Jordan algebras} } }%

 \medskip

\centerline{ Fernando Montaner \footnote{Partially supported  by
the Spanish Ministerio de Ciencia y Tecnolog\'{\i}a and FEDER (MTM2010-18370-C04-02), and by the
Diputaci\'on General de Arag\'on (Grupo de Investigaci\'on de
\'Algebra).}} \centerline{{\sl Departamento de Matem\'{a}ticas,
Universidad de Zaragoza}} \centerline{{\sl 50009 Zaragoza,
Spain}} \centerline{e-mail: fmontane@unizar.es}

 \medskip

%
%
\begin{abstract}
We address a Jordan version of Johnson theorem on (associative) algebras of quotients, namely whether a strongly nonsingular (the Jordan version of nonsingularity) has a von Neumann regular algebra of quotients.

Although the answer is negative in general, we prove that the result holds in the important case of algebras satisfying a polynomial identity.

\end{abstract}

\section*{Introduction} Algebras of quotients of Jordan algebras have received quite some attention in recent times (see \cite{densos, localorders} and references therein for a succinct survey of some of the work devoted to them). As in the case of associative algebras, from which that study draws much of its themes an techniques, it all began with the search of classical algebras of fractions (and more precisely, analogues of Goldie's theorems obtained by Zelmanov, see \cite{z-goldie-1, z-goldie-2}) thus answering a question derived from the fundamental problem posed by Jacobson in \cite[p. 426]{jac-ams} on Ore-like conditions on Jordan algebras ensuring the construction of algebras of fractions (which, in its original version, was answered by Mart\'{\i}nez in \cite{martinez}). A first wide generalization of those constructions, as was the case for algebras of quotients of associative algebras, was considering algebras of quotients of nonsingular Jordan algebras, for an adequate notion of nonsingularity in the Jordan context (a notion that was termed strong nonsingularity,  see \cite{esenciales} and definition below). That construction, which is a Jordan analogue of Johnson's algebra of quotients, made use of the filter of essential inner ideals, adapting to Jordan algebras one of the associative constructions of Jonhson's algebra of quotients, namely the one that does not rely on the injective hull of the algebra, which due to the non linearity of the notion inner ideal, is a concept not available for Jordan algebras (see however the comment at the beginning of 4.7 of \cite{pi-ii}).

As in the associative setting, Johnson's algebras of quotients can be viewed in the framework of general algebras of quotients as defined in \cite{densos}, since they are those for which essential inner ideals are dense (see \cite[5.19] {densos}).

A  fundamental result on associative algebras of quotients is Johnson's theorem (\cite[13.36]{Lam}, \cite[Theorem 26.11]{passman}, \cite[Theorem 26.11]{bmami}), which asserts that the maximal (or Utumi-Lambek) algebras of quotients of an associative algebra is von Neumann regular if and only if the algebra is nonsingular, so that its maximal algebra of quotients coincides with its Johnson algebra of quotients.

 Our purpose in the present paper is to examine whether a Jordan analogue of that result holds since as the title of the paper reads,  we are interested in conditions under which algebras of quotients of Jordan algebras are von Neumann regular. Certainly strong nonsingularity as defined in \cite{esenciales} which is the natural working  generalization of nonsingularity in the Jordan context, is not the only possible such condition (and as we will se later is not always even a condition for that). We can mention other conditions, as being Goldie \cite{fgm}, or being local Lesieur-Croisot \cite[Theorem 5.5]{localorders}, but all those conditions imply strong nonsingularity.

 The paper is organized as follows: After a first section of preliminaries which contains basic general notation and definitions, including a comment
 on polynomial identities on quadratic Jordan algebras to be generalized in section 4 to basic notions on polynomial identities on quadratic algebras, following the exposition provided in \cite[3.1.5]{jac-struc}, we recall the notion of general quotients, as defined in \cite{densos} as well the construction of maximal algebras of quotients.

 We then move to the answering of the question posed before, namely whether a strongly nonsigular Jordan algebra is von Neumann regular. We give on section 3 a counterexample for the general assertion.

 In the next section we start the study of that question in the particular but quite central case in Jordan theory where the algebra is nondegenerate. PI.

 Due to the form of the construction of the maximal algebra of quotients in that situation, which relies strongly on the module structure over its centroid, whose central role in the study of PI algebras is due to \cite[3.6]{fgm} for Jordan algebras of the Posner-Rowen theorem (see \cite[6.1.28]{anillos}, \cite[1.7.9]{rowen-pi}), which asserts that any nonzero inner ideal of a nondegenerate Jordan algebra hits non trivially the weak center (\cite{fulgham}).

 We devote the next section to exposing some facts on the centroid of the algebras of quotients of nondegenerate PI-algebras, generalizing to that situation the results of \cite{pi-covers} (see also [5.19.(4)] \cite{densos}). The next section follows the ideas of Beidar \cite{bmi}, and in particular the way he devised of circumventing the obstacle which arises from the fact that subdirect product decomposition provided by the set of all the strongly prime quotients of a nondegenerate algebra may not interact well with the construction of algebras of quotients.

 We quote \cite{lw}, who remark when commenting on Beider's (and Mikhalev's) work on orthogonal completeness that while trying to relate polynomial identities of a prime ring $R$ to polynomial identities os its ring of quotients $Q(R)$ is quite direct, the approach to the same problem for semiprime rings suggested by the decomposition of such an $R$ as a subdirect product of its prime quotients $R/P$ faces the obstacle of the not existence of homomorphisms $Q(R)\rightarrow Q(R/P)$ in general.

 Following Beidar's approach we overcome this problem by the theory of orthogonal completeness and its attached subdirect decomposition.

 These ideas are introduced and applied in sections 6 and 7. In this later section we prove our main result, namely the answer in the affirmative of the above mentioned question on von Neumann regularity of algebras of quotients in the case of PI-algebras.

That approach was also followed by Artacho C\'{a}rdenas et al. in \cite{agr1} and \cite{agr2}.

%
%

\section{Preliminaries}\label{pre}

\begin{apartado}


  We will work with  Jordan  algebras over a unital commutative ring of scalars $\Phi$
   which will be fixed throughout. We refer to
\cite{jac-struc, mcz} for   notation, terminology and basic
results on Jordan algebras. We will occasionally make use of some results on  Jordan pairs and on Jordan triple systems, mainly
obtained from algebras. Our basic reference for Jordan pairs (and of basic notions of triple systems) is \cite{loos-jp}. In particular, we make use of the identities proved in
\cite{jac-struc} and \cite{loos-jp}, which will be quoted with the
labellings QJn and JPn respectively. In addition to that, we will often rely on an associative background, both as an ingredient of Jordan theory when dealing with special Jordan algebras, and as a source of notions which have been extended to the Jordan algebra setting, among which we will mainly consider those  from localization theory, for which we refer to \cite{st}.

We next recall some
of those basic results  and notations together with some other results  that will
be used in the paper.
\end{apartado}

\begin{apartado}A Jordan algebra has products   $U_xy$ and $x^2$, quadratic in $x$
and linear in $y$, whose linearizations are
 $U_{x,z}y=V_{x,y}z=\{x,y,z\}=U_{x+z}y-U_xy-U_zy$ and $x\circ y=V_xy=(x+y)^2-x^2-y^2$.

 We will denote by $\widehat{J} $ the free unital hull $\widehat{J}=\Phi 1\oplus
 J$ with products $U_{\alpha
 1+x}(\beta1+y)=\alpha^2\beta1+\alpha^2y+\alpha x\circ
 y+2\alpha\beta x+\beta x^2+U_xy$ and
 $(\alpha1+x)^2=\alpha^21+2\alpha x +x^2$.

It is well known that  any associative algebra $A$ gives rise to a Jordan algebra $A^+$ with products $U_xy=xyx$ and $x^2=xx$. A Jordan algebra is special if it is isomorphic to a subalgebra of an algebra $A^+$ for an associative $A$. If $A$ has an involution $\ast$ then $H(A,\ast)=\{a\in A\mid a=a^\ast\}$ is a Jordan subalgebra of $A^+$ and so are ample subspaces $H_0(A,\ast)$ of symmetric elements of $A$, subspaces such that $a+a^\ast, aa^\ast$ and $aha^\ast$ are in $H_0(A,\ast)$ for all $a\in A $ and all $h\in H(A,\ast)$.
\end{apartado}

\begin{apartado} A $\Phi$-submodule $K$ of a Jordan algebra $J$ is
an \emph{inner ideal} if $U_k\widehat{J}\subseteq K$ for all $k\in
K$. An inner ideal $I\subseteq J$ is an \emph{ideal} if
$\{I,J,\widehat{J}\}+U_JI\subseteq I$, a fact that we will denote in the usual way $I\triangleleft J$. If $I$, $L$ are ideals of
$J$, so is their product $U_IL$  and in particular   the
\emph{derived ideal} $I^{(1)}=U_II$ of $I$. An (inner) ideal  of $J$ is
\emph{essential}\  if it has nonzero intersection  with any
nonzero (inner) ideal of $J$.

A Jordan algebra $J$ is \emph{nondegenerate} if $U_xJ\neq0$ for any nonzero $x\in J$, and \emph{prime} if $U_IL\neq0$ for any nonzero ideals $I$ and $L$ of $J$. The algebra $J$ is said to be \emph{strongly prime} if $J$ is both nondegenerate and prime.
\end{apartado}

\begin{apartado}\label{annihilator_def}
If $X\subseteq J$ is a subset of a Jordan algebra $J$, the
\emph{annihilator} of $X$ in $J$ is the set $ann_J(X)$ of all
$z\in J$ such that $U_xz=U_zx=0$ and
$U_xU_z\widehat{J}=U_zU_x\widehat{J}=V_{x,z}\widehat{J}=V_{z,x}\widehat{J}=0$
for all $x\in X$. The annihilator  is always an inner ideal of $J$,
and it is an ideal if $X$ is an ideal. If $J$ is a nondegenerate
Jordan algebra  and $I$ is an ideal of $J$, then the annihilator
of $I$ in $J$ can be characterized as follows
(\cite{mc-inh,pi-ii}):
$$ann_J(I)=\{z\in J\ \mid\ U_zI=0\}=\{z\in J\ \mid\ U_Iz=0\}.$$
\end{apartado}

\begin{apartado}For a set $X$ and elements $x_1,\dots,x_n\in X$, an element $f(x_1,\dots,x_n)$ of the free associative algebra $FAss[X]$ is called admissible if at least one of its monomials of highest degree is monic \cite[p. 109]{zsss}, \cite{zazd}.

Recall (\cite[Chapter 3]{jac-struc}, \cite[IV.B]{taste}) that an element $f(x_1,\dots,x_n)$ of the free Jordan algebra $FJ[X] $ on a countable set of generators $X=\{x_1,x_2,x_3,\dots\}$ is an identity of the Jordan algebra $J$ (or equivalently that $J$ satisfies $f(x_1,\dots,x_n)=0$) if it vanishes for each substitution $f(a_1\dots,a_n)$ of the variables $x_i$ by elements $a_i$ of $J$. The algebra $J$ strictly satisfies the identity $f$ if $f$ is an identity of each scalar extension $J_\Omega = J\otimes_\Phi\Omega$ of $J$ for a commutative associative algebra $\Omega \supseteq \Phi$.

An element $f$ of the free associative algebra on the set $X$ is admissible if it has a monic leading term (\cite[p. 102]{zsss}). An element $f(x_1,\dots,x_n)$ of  $FJ[X] $ is essential if its image in the free special Jordan algebra $FSJ[X]$ under the natural projection is admissible as an associative polynomial, that is as an element of the free associative algebra on $X$ (\cite[p. 112]{zsss}, \cite{zazd}). In addition, the element $f(x_1,\dots,x_n)$ is essential if its image under the natural projection in the free special Jordan algebra $FSJ[X]$ is admissible as an associative polynomial, that is as an element of the free associative algebra on $X$.

A Jordan algebra $J$ is a PI-Jordan algebra if there exists an essential Jordan polynomial vanishing strictly on $J$ (\cite{mcz}). Note that if there exists an essential polynomial $f$ vanishing on $J$ (not necessary stritly), then there exists a multilinear essential polynomial vanishing on $J$, since the multilinearization procedure \cite[p.14]{zsss} can be applied to $f$ and simultaneously to its image in $FSJ[X]$, hence to the associative polynomial corresponding to $f$, which has a monic leading term. Therefore there exists a multilinear essential polynomial which strictly vanishes on $J$ (\cite[1.5]{zsss}), and therefore $J$ is PI.

\end{apartado}

\begin{apartado}\label{local algebras} For any element  $a$ in a Jordan algebra $J$,  the   \emph{local algebra  $J_a$
of $J$ at $a$} is the  quotient of the $a$-homotope $J^{(a)}$, the Jordan algebra defined over the $\Phi$-module $J$ with operations $U^{(a)}_xy = U_xU_ay$ and  $x^{(2,a)}= U_xa$, by the
ideal $Ker\,  a$ of $J^{(a)}$ of all the elements $x\in J$ such that
$U_ax=U_aU_xa=0$. If $J$ is nondegenerate  the above conditions on $x$
reduce  to $U_ax=0$.
Local algebras at nonzero elements of a nondegenerate (resp. strongly prime) Jordan algebra are themselves nondegenerate (resp. strongly prime) \cite[3.1.5]{jac-struc}. (We recall that similar definitions can be given for associative algebras, for which we will also use the notation $A
_x$ for the local algebra at an element $x$ of the $A$.)
\end{apartado}

 \section{Algebras of quotients}

As mentioned in the introduction, the study of algebras of quotients of Jordan algebras draws its inspiration from  associative theory  (see \cite{densos}). We recall next some basic notation from the latter and refer the reader to \cite{anillos,st} for basic results about algebras of
 quotients for associative algebras.

\begin{apartado}Let $L$ be a left ideal of an associative algebra $R$. Recall the usual associative
notation  $(L:a)$, with $a\in R$, for the
 set of all elements $x\in R$ such that $xa\in L$ (see, for instance, \cite{st}). A left ideal $L$
 of $R$ is \emph{dense} if $(L:a)b\neq0$
 for any $a\in R$ and any nonzero $b \in R$.
\end{apartado}

\begin{apartado} The associative algebras naturally arising in Jordan theory are associative
 envelopes of Jordan algebras, and therefore   carry an involution. That makes important
 to be able to extend involutions to their algebras of quotients. The fact that this is not always possible for the one-sided maximal algebras
 of quotients $Q_{max}^r(R)$ and $Q_{max}^l(R)$ leads to the  use of the maximal symmetric algebra of quotients
 $Q_\sigma(R)$ (see \cite{lanning}) as an adequate
 substitute of that algebra. Recall that $Q_\sigma(R)$  is
 the set of elements $q\in Q_{max}^r(R)$ for which there exists a
 dense left ideal $L$ of $R$ such that $Lq\subseteq R$
 (which up to a canonical isomorphism can be viewed  symmetrically as the set of all $q\in Q_{max}^l(R)$ for which there exists a
 dense right ideal $K$ of $R$ such that $qK\subseteq R$).  If $R$ has
 an involution,   $Q_\sigma(R)$ is the biggest subalgebra of $Q_{max}^r(R)$ and
 $Q_{max}^l(R)$ to which that involution extends.
 \end{apartado}

\begin{apartado}\label{new_ii} Let $J$ be a Jordan algebra, $K$ be  an inner ideal of $J$ and $a\in J$.
Following
  \cite{densos, esenciales} we set $$ (K:a)_L=\{x\in K\  \mid   x\circ a\in K  \},$$
$$ (K:a)=\{x\in K\  \mid U_ax,\ x\circ a\in K  \}.$$
It is straightforward to check that both $ (K:a)_L$ and $ (K:a)$ are  inner ideals of $J$ for
all $a\in J$, and that in addition, the containment $U_{(K:a)_L}K\subseteq (K:a)$   holds \cite[Lemma 1.2]{densos}. Given any  finite family of elements $a_1,\ \ldots, a_n\in
J$, we inductively define $(K:a_1:a_2:\ldots:a_n)=((K:a_1:\ldots
: a_{n-1}): a_n)$.\end{apartado}

\begin{apartado}\label{nii} An inner ideal $K$ of $J$ is said to be \emph{dense} if
$U_c(K:a_1:a_2:\ldots:a_n)\neq0$ for any finite collection of
elements  $a_1,\ \ldots, a_n\in J$ and any $0\neq c\in J$. Different
characterizations of density are given in \cite[Proposition 1.9]{densos}. Recall that
if $K$ is a dense inner ideal of $J$ so are the inner ideals  $(K:a)$
for all $a\in J$ \cite[Lemma 1.8]{densos}.
\end{apartado}

\begin{apartado}\label{denominadores} Let $\widetilde{J}$ be a Jordan algebra,  $J$  be  a subalgebra of
$\widetilde{J}$, and   $\widetilde{a}\in\widetilde{J}$. Recall from
\cite{pi-ii} that an element $x\in J$ is a \emph{$J$-denominator} of
$\widetilde{a}$ if the following multiplications take
$\widetilde{a}$
 back into $J$:

 \centerline{ \begin{tabular}{lll}
   (Di)\ $U_x\widetilde{a}$ &  (Dii)\ $U_{\widetilde{a}}x$ &  %
   (Diii)\ $U_{\widetilde{a}}U_x\widehat{J}$ \\
   (Diii')\ $U_xU_{\widetilde{a}}\widehat{J}$  & %
   (Div)\ $V_{x,\widetilde{a}} \widehat{J}$  & (Div')\ $V_{ \widetilde{a},x} \widehat{J}$ \\
 \end{tabular}}

 \noindent We will denote the set of $J$-denominators of $\widetilde{a}$ by
 ${\mathcal D}_J(\widetilde{a})$. It has been proved in \cite{pi-ii}
 that ${\mathcal D}_J(\widetilde{a})$ is an inner ideal of $J$. Recall also from  \cite[p. 410]{fgm} that any $x\in J$ satisfying (Di),
 (Dii), (Diii) and (Div) belongs to ${\mathcal D}_J(\widetilde{a})$. The following procedure for obtaining denominators is given in \cite[Lemma 2.2]{fgm}, and has the advantage of being `context free", that is of not depending of the overalgebra $\widetilde{J}$: for any $x\in J$, the containments $x\circ \widetilde{a}, U_x\widetilde{a}\in J$ imply $x^4\in {\mathcal D}_J(\widetilde{a})$.
\end{apartado}

\begin{apartado}\label{aq} Let $J$ be a subalgebra of a Jordan algebra $Q$. Following \cite{densos}, we say that $Q$ is a \emph{general algebra of quotients of $J$} if the
following conditions hold
\begin{enumerate}\item[(AQ1)] ${\mathcal D}_J(q)$ is a dense
inner ideal of $J$ for all $q\in Q$. \item[(AQ2)] $U_q {\mathcal
D}_J(q)\neq0$ for any nonzero $q\in
Q$.\end{enumerate}

Note that any nondegenerate Jordan algebra is its own algebra of
quotients. Conversely any Jordan algebra having an algebra of
quotients is nondegenerate.

A different, though closely related approach to algebras of quotients was carried out in \cite{esenciales} by using essential inner ideals as sets of denominators. That second approach, which in fact motivated and inspired the one in \cite{densos} as well as some other treatments of algebras of quotients in Jordan algebras (see the references in \cite{densos}), has the advantage that checking essentiality is significantly simpler than checking density. However, for that choice to be feasible, the additional condition of strong nonsingularity, introduced in \cite{esenciales}, is needed. We will  comment on that at the end of the present section.\end{apartado}

\begin{apartado}\label{def}  An algebra  of quotients $Q$ of a Jordan algebra $J$
is said to be a \emph{maximal  algebra of quotients} of $J$ if for
any other algebra of quotients $Q'$ of $J$ there exists a
homomorphism $\alpha: Q' \to Q$ whose restriction to $J$ is the
identity map.
\end{apartado}

\begin{apartado} \label{ampleher}

Maximal algebras of quotients of nondegenerate Jordan algebras have been proved to exist and described in \cite{densos}.  We recall now some of the features of that construction.

Following \cite[4.9]{densos}, we first consider the Zelmanov polynomial  in the variables $X_{(i)}= \{x_i,y_i,z_i,t_i\}$ $(i=1,2,3)$ defined in \cite[pp.192-195]{mcz} given by $$Z_{48}= [[P_{16}(X^{(1)}), P_{16}(X^{(2)})], P_{16}(X^{(3)})]$$
for $P_{16}(X) = [[[t,[t,z]]^2,[t,w]],[t,w]]$  $(t= [x,y])$, where $[[a,b]c]= \{a,b,c,\}-\{b,c,a\}$. Denote now by ${\cal Z}(X)$ the hermitian ideal of the free Jordan algebra $FJ[X]$ over a countable set of generators $X$ generated by all the evaluations $Z_{48}(a_1,\dots, a_{12})$ for $a_i\in FJ[X]$.

Together with the hermitian case where the algebra $J$ has $ann_{J}({\cal Z}(J))=0$, we consider algebras $J$ that satisfy ${\cal Z}(J)=0$, an therefore are PI. These two extremes appear in a decomposition of the algebra as a subdirect sum related to the ideal ${\cal Z}(J)$. Let $J$ be a nondegenerate Jordan algebra. Then:
\begin{enumerate}\item There are two strongly prime ideals $P_1, P_2\triangleleft J$ that satisfy:
 $P_i= ann_J(P_j)$, ($\{i,j\}= \{1,2\}$, $J_1= J/P_1$ is a PI-algebra, and  $J_2=J/P_2$ is a special Jordan algebra such that $ann_{J_2}({\cal Z}(J_2))=0$.
 \item The maximal algebra of 	quotients of $J$ is a direct product $Q_{max}(J)\cong Q_{max}(J_1)\boxplus Q_{max}(J_2)$, where
 \item $Q_{max}(J_1) = (J_1)_{{\cal E}(\Gamma(J_1))}$ is the almost classical algebra of quotients: the module of $\E(\Gamma(J_1))$-quotients, $(J_1)_{\E(\Gamma(J_1))} = \displaystyle\lim_{\longrightarrow}\{\Hom({\frak a}, J_1)\mid {\frak a} \in \E(\Gamma(J_1))\}$, for the filter $\E(\Gamma(J_1))$ of essential ideals of the nonsingular ring $\Gamma(J_1)$, the centroid of $J_1$, endowed with the natural Jordan algebra structure (see \cite[Theorem 3.11]{densos})
 \item $Q_{max}(J_2) = \{q\in H(Q_\sigma(R),\ast) \mid {\cal D}_J(q)  \; \hbox{\rm is dense in} \; J_2 \}$ is an ample subspace $H_0(Q_\sigma(R),\ast)$ in the set of symmetric elements of the symmetric algebra of quotients $Q_\sigma(R)$ of a $\ast$-tight enveloping algebra $R$ of $J_2$

\end{enumerate}

Maximal algebra of quotients thus defined provided an unified framework for other classes of algebras of quotients found in the literature (see \cite{densos} and references therein). In particular, the algebra of quotients obtained when taking as a filter of denominator inner ideals the set of essential inner ideals, which provides a Jordan version of Johnson's algebra of quotients, can be defined whenever that set of inner ideals is a linear topological filter (see \cite[12]{esenciales}), and this is a consequence of strong nonsingularity, the Jordan generalization of nonsingularity: a Jordan algebra $J$ is strongly nonsingular if for any essential inner ideal $K$, and any $a\in J$, $U_{a}K=0$ implies $a=0$ \cite[1.3]{esenciales}. For a strongly nonsingular Jordan algebra an inner ideal is essential if and only if it is dense \cite[1.10]{densos} (the reciprocal, namely that if a nondegenerate Jordan algebra essential inner ideals and dense inner ideals coincide then the algebra is strongly nonsingular, is obvious).

\end{apartado}

\section{A counterexample}

We next address the problem posed in the introduction. Concretely we examine whether the maximal algebra of quotients of a strongly nonsingular Jordan algebra is von Neumann regular.

Since the maximal algebra of quotients of a strongly nonsingular Jordan algebra is its Johnson algebra of quotients, the problem amounts to examining whether the Johnson algebra of quotients of a strongly nonsingular Jordan algebra is von Neumann regular, or still, whether the maximal algebra of quotients of a Jordan algebra is von Neumann regular provided that inner ideals are dense if and only if they are essential.

Our first result consists of showing that the answer to that question is negative in general. However we will afterwards give an important instance at which the answer is positive, namely the case of PI-algebras.

Consider $J=H(D,\ast)$ for a domain $D$ with involution $\ast$, which is obviously strongly nonsingular.

We describe $Q_{max}(H(D,\ast))$.

In case $Z(H(D,\ast))\neq0$, $Q_{max}(H(D,\ast))$ can be described as follows according to \ref {ampleher} (4):

$$Q_{max}(H(D,\ast))=H(Q_{\sigma}(D),\ast)$$

where $Q_{\sigma}(R)$ for an associative ring $R$ with involution $\ast$ is the Maximal algebra of symmetric quotients \cite{lanning}: the set of all $q\in Q_l(R)$ such that $qK \subseteq R$ for some essential right ideal $K$ of $R$, which is the biggest subalgebra of $Q_l(R)$ to which $\ast$ can be extended.

\begin{lema} \label{regularisinvertible} $Q_{\sigma}(D)$ is a (unital) domain, hence von Newmann regular elements of $Q_{\sigma}(R)$ are invertible.\end{lema}

\begin{proof} The first assertion is obvious, and for the second note that if $d \in D$ has $ded=d$ and $ede=e$ for some $e \in D$, hence $dede=(ded)e=de$ hence $de(de-1)=0$, thus $de=1$ since $D$ is a domain. Analogously $ed=1$, hence $d$ is invertible. \end{proof}

As a particular case of that situation we obtain the counterexample we are looking for. Take $D=Ass[X,Y]$ the free associative algebra of two generators $X,Y$ with the standard involution, so that $H(D,\ast)=FSJ[X,Y]$ is the free special Jordan algebra on two generators.

\begin{proposicion} For the case $(D,\ast)$, $J=H(D,\ast)$ is strongly nonsingular, and it is not von Newmann regular\end{proposicion}

\begin{proof} Note first that  $J=H(D,\ast)$ is a domain, hence $J$ is a strongly nonsingular (any nonzero left ideal has zero annihilator). Now if $Q_{\sigma}(J)$ is von Neumann regular, $X \in D \subseteq J$ is von Newmann regular, hence it is invertible in $Q\sigma (D)$ as seen before \ref{regularisinvertible}. Therefore there is an essential left ideal $L$ of $D$ with $LX^{-1} \subseteq D$, hence $L \subseteq DX$, but since $L$ is essential so is $DX$ hence $DX \cap DY \neq 0$, which obviously does not hold for $D=Ass[X,Y]$, a contradiction. \end{proof}

\section{Centroid}

As mentioned in the outline of the contents of de paper in section 1, the module structure of a Jordan algebra over its centroid provides the tools that allow the study of algebras of quotients of PI Jordan algebras through the localization with respect to the filter of essential ideals. In this section we deal with the centroid of Jordan algebras and their algebras of quotients, the localization of these rings and the Jordan algebras as modules over these rings, and their relation to the extended centroid of quadratic Jordan algebras.

We begin with the following fact that shows that the maximal ring of quotients is a ``idempotent" construction.

The notations we make use of are the ones introduced in \cite{densos}.

\begin{lema}\label{maxclosure}
If $J$ is  a nondegenerate Jordan algebra, then $Q_{max}(Q_{max}(J))
= Q_{max}(J)$.\end{lema}

\begin{proof} Set $Q = Q_{max}(J)$, and let $\tilde Q\supseteq Q$ be an algebra of quotients of
$Q$. By the transitivity of algebras of quotients, $\tilde Q$ is an algebra of quotients of
$J$. Therefore, there exists a $J$-homomorphism (i.e. a
homomorphism that restricts to the identity mapping on $J$)
$\alpha: \tilde Q \rightarrow Q$ by the maximality of $Q$. Denote by $\iota: Q\rightarrow \tilde Q$ the
inclusion. By the uniqueness \cite[Lemma 2.11]{densos} of $J$-homomorphisms of algebras of quotients, we have $\iota\alpha = {\rm Id}_{\tilde
Q}$, and $\alpha\iota= {\rm Id}_Q$. In particular, $\iota$ is surjective, hence $Q= \tilde
Q$.\end{proof}

\begin{apartado}Note that the previous lemma is valid for $\Phi$-algebras, algebras over a fixed ring of
scalars $\Phi$. On the other hand, the proof of the fact that $Q_{max}(J) = J_{{\mathcal E}(J)}$ for a
Jordan PI-algebra $J$ shows that this is the maximal algebra of quotients when $J$ is
considered as a $\Phi$-algebra over any ring of scalars $\Phi$  which makes $J$ a $\Phi$-algebra (that is, for which the multiplication by elements of $\Phi$ defines an operator of the centroid). In particular this is the
case for $\Phi = {\bf Z}$, the ring of integers. Now, since any Jordan $\Phi$-algebra of
quotients is a ${\bf Z}$-algebra of quotients, we get that the fact that $J_{\mathcal E(J)}$ is the
maximal algebra of quotients of $J$ is independent of the ring of scalars. Moreover, this
also shows that $(J_{\mathcal E(J)})_{{\mathcal E}(J_{\Gamma_{\mathcal E}(J)})}= J_{\mathcal E(J)}$. Note however that we may have $\Gamma_{{\mathcal E}(J)}
\ne \Gamma(J_{{\mathcal E}(\Gamma)}) = \Gamma(J_{{\mathcal E}(\Gamma)})_{{\mathcal E}(\Gamma_{{\mathcal E}(J)})}$.

In the associative theory it is well known the centroids of the maximal rings of quotients (as well as other rings of quotients) of a semiprime ring coincide with its extended centroid. For quadratic Jordan algebras, the definition of extended centroid was introduced in \cite{pi-ii}. Here we relate the extended centroid of a quadratic Jordan PI algebra to the centroid of its maximal ring of quotients.

\end{apartado}

\begin{lema}Let $J$ be a nondegenerate PI Jordan algebra. Then $\Gamma_{{\mathcal E}(J)}
=\C(J)$.\end{lema}

\begin{proof} Take an element $\gamma \in \Gamma_{\mathcal E}$, and a representative $[f,\frak a]\in \gamma$. Then $I =\mathfrak
a J$ is an essential ideal of $J$ by \cite[3.7]{densos}. Define $\tilde f: I \rightarrow J$ by
$\tilde f(\sum_i \alpha_i x_i) = \sum_if(\alpha_i)x_i$ for $\alpha_i \in \mathfrak a$ and $x_i\in J$.
To see that $\tilde f$ is well defined, suppose that $\sum_i \alpha_i x_i= 0$ for some
$\alpha_i \in \mathfrak a$ and $x_i\in J$, and set $b = \sum_if(\alpha_i)x_i$. Then, for any
$\alpha \in \mathfrak a$ we have $\alpha b=\alpha \sum_if(\alpha_i)x_i = \sum_if(\alpha_i\alpha)x_i
=\sum_if(\alpha)\alpha_ix_i= f(\alpha)\sum_i\alpha_ix_i= 0$, hence $\mathfrak ab = 0$, and
thus $b = 0$ by [M, 3.6(2)]. Now it is easy to see that $\tilde f \in \Hom_J(I,J)$ (note
that $\tilde f(U_{\alpha x}y) = \tilde f(\alpha^2U_xy)= \alpha f(\alpha)U_xy\in \mathfrak aJ=I$).
Moreover, if $\gamma$ admits the two representatives $(f,\mathfrak a)$, and $(g,\mathfrak b)$, a straightforward verification shows that
$\tilde f = \tilde g$ on $(\mathfrak a\cap \mathfrak b)J$, which is an essential ideal of $J$ to which
both $\tilde f$ and $\tilde g$ restrict. Thus, the rule $[f,\mathfrak a] \mapsto [\tilde f, \mathfrak
aJ]\in \C(J)$ defines a mapping $\phi:J_\E\rightarrow \C(J)$ which is easily seen to be a
homomorphism.

Now take $\gamma  \in \C(J)$ , and set $\gamma = [f,I]$ for some essential $I\triangleleft J$, and $f
\in \Hom_J(I,J)$. Define $\tilde f: (I:J)_\Gamma\rightarrow \Gamma$ by $\tilde
f(\delta)= f\delta$ (where juxtaposition stands for composition). Note that since $\delta(J)\subseteq I$ for any
$\delta \in \mathfrak (I:J)_\Gamma$, $f$ is defined on the image of $\delta$, hence $\tilde f$ is well defined. Now, since $(I:J)_\Gamma$
is essential, this defines an element $[\tilde f, (I:J)_\Gamma]\in \Gamma_\E$. Moreover, it
is easy to see that the class $[\tilde f, (I:J)_\Gamma]\in \C(J)$ does not depend of the
selected representative $(f,I)\in \gamma$, and therefore, the correspondence $[f,I]\mapsto
[\tilde f, (I:J)_\Gamma]$ defines a mapping $\psi: \Gamma_\E\rightarrow \C(J)$.

Take $\gamma =[f,\mathfrak a] \in \Gamma_\E$, and consider $\phi(\gamma) = [\tilde f,{ \mathfrak a} J]$
as before. Then $\mathfrak a \subseteq ({\mathfrak a}J:J)_\Gamma$, and for any $\alpha \in{ \mathfrak a}$ we have
$\psi(\phi(\gamma)) = [\tilde{\tilde f}\;,({\mathfrak a}J:J)_\Gamma] = [\tilde{\tilde f}\;,\mathfrak a]$
where $\tilde{\tilde f}\;(\alpha) = \tilde f\alpha$. Now, if
$x\in J$, we have $(\tilde f\alpha) x = \tilde f(\alpha x) = f(\alpha)x$, hence
$\tilde{\tilde f\;}(\alpha) = f(\alpha)$, and thus $\psi\phi = {\rm Id}_{\Gamma_\E}$. On
the other hand, if $\gamma = [g, I] \in \C(J)$, we have $\psi(\gamma) = [\tilde g,
(I:J)_\Gamma]$ as before, and $\phi(\psi(\gamma)) = [\tilde{\tilde g},(I:J)_\Gamma J]$,
where $\tilde{\tilde g}(\alpha x) = \tilde g(\alpha)x = (g\alpha)(x) = g(\alpha x)$ for any
$x\in J$ and any $\alpha \in (I:J)_\Gamma$. Thus $\tilde{\tilde g} = g$ on $(I:J)_\Gamma
J$, hence $\phi(\psi(\gamma)) = [g, I] =\gamma$, and we get $\phi\psi= {\rm
Id}_{\C(J)}$.\end{proof}

\begin{lema}\label{centr}Let $J$ be a nondegenerate Jordan algebra. Then $\Gamma(Q_{max}(J)) =
\C(Q_{max}(J))$.\end{lema}

\begin{proof} Set $Q = Q_{max}(J)$. Since $\C(Q)Q$ is an algebra of quotients of $Q$,
lemma \ref{maxclosure}
implies that $\C(Q)Q = Q$, hence $\C(Q) \subseteq \Gamma(Q)$.\end{proof}

Note that this result does not strictly parallels the result for associative algebras, for which the centroid of the maximal algebra of quotients and the extended centroid of the algebra coincide. As for the usual central closure for quadratic Jordan algebras, the usual construction may not be centrally closed. However, an inductive transfinite procedure gives rise to a centrally closed algebra (see \cite[4]{mcz}). In our situation we can follow a similar argument to construct a ``closed'' extended centroid $\cal {C}_{\infty}(J)$ (see \cite{pi-ii}). With this notion we propose:

\begin{apartado}{\bf Conjecture:} Define $\C_\infty(J)$ as in \cite[ 4.11]{pi-ii}. Then
$\Gamma(Q_{max}(J)) = \C_\infty(J)$. \end{apartado}

\section{Polynomial identities and algebras of quotients}

In this section we prove that an algebra of quotients of a PI-Jordan algebra $J$ inherits the vanishing
of Jordan polynomials on $J$. This result is essentially contained in \cite{densos}. Recall \cite[3.1.5]{jac-struc} that
a unital quadratic algebra over $\Phi$ is a triple $(J,U,1)$ where $J$ is a $\Phi$-module, $1 \in J$, an $U$ is a quadratic map
$J\rightarrow \End_{\Phi}(J)$ such that $U_1=\rm Id$. More generally, a quadratic algebra (not necessarily unital) is a pair
$(J,U)$ such that there is a quadratic unital algebra $(\widehat{J}, \widehat{U},1)$ such that $J \subseteq \hat{J}$ is a quadratic subalgebra
in the obvious sense.

We refer to \cite{jac-struc} for the definition of the free quadratic $\Phi$-algebra $FQ[X]$ over a set $X$ of generators.
Its elements are called quadratic polynomials and it satisfies the universal property of uniquely extending any mapping
$\alpha: X \rightarrow J$ into a quadratic algebra $J$ to a homomorphism of quadratic algebras $\tilde{\alpha}: FQ[X]\rightarrow J$.
For an element $f(x_1, \dots, x_n)\in FQ[X]$ we write $\tilde{\alpha}(f)=f(\alpha(x_1),\dots,\alpha(x_n))$, which is called the
evaluation of $f$ at $x_i=\alpha(x_i)$.

For us the most important class of quadratic algebras is the variety (in the universal algebras sense) of quadratic Jordan algebras,
for which $U$ is the usual quadratic mapping.

Assume that $\Phi$ is a semiprime commutative associative ring, and let ${\cal E}={\cal E} (\Phi)$ its filter of essential
ideals. Suppose that $Q$ is a quadratic $\Phi$-algebra which is a nonsingular $\Phi$-module. As in \cite{densos},
we can follow the construction of the nearly classical rings of fractions \cite{bmi} (named almost-classical rings of quotients
in \cite{wis,densos}) we give $Q_{\cal E}$ a structure of quadratic algebra. Since $ Q_{\cal E}=\displaystyle \lim_{\underset{\frak a \in \cal E}\longrightarrow} \Hom_{\Phi}(\frak a, Q)$, we can represent an element $q \in Q_{\cal E}$ as an equivalence class $q=[f,\frak a]$ where $\frak a \in \cal E$, $f: \frak a \rightarrow Q$ is a homomorphism, and $[f,\frak a]=[g,\frak b]$ for $\frak b \in \cal E$, $g \in \Hom _{\Phi}(\frak b, Q)$ if the restrictions $f \vert_{\frak a \cap \frak b}$ and $g \vert_{\frak a \cap \frak b}$ coincide (see \cite{st}).

For any $n\ge 1$ and $\frak a \in \cal E$, we have $$\frak a^{[n]}=\sum _{\alpha \in \frak a} \alpha^{n}\Phi \in \cal E,$$
and if $\frak {a}_1,\dots,\frak {a}_k \in \cal E$, $\frak {a}_1\cdots\frak {a}_k \in \cal E$, $Q_{\cal E}$ can be given a structure
of quadratic algebra in the following way: if $q,p \in Q_{\cal E}$, and $q=[f,\frak a]$, $p=[d,\frak b]$, we set $U_{q}P=[h,\frak {a}^{[2]} \frak b]$,
where $h:\frak {a}^{[2]} \frak b \rightarrow Q$ is determined by $h(\alpha ^{2} \beta)=U_{f(\alpha)}g(\beta)$ for $\alpha \in \frak a$,
$\beta \in \frak b$.

Since a unital quadratic algebra not only has the quadratic operator $U$ but also the unit element $1$ which must satisfy $U_1=\rm Id$, and
a quadratic algebra embeds in a unital quadratic algebra, to prove that we have obtained a quadratic algebra structure on $Q_{\cal E}$
we need to embed it in a unital quadratic algebra $\tilde{Q}_{\cal E}$ in the sense that quadratic  map of $Q_{\cal E}$ is the restriction
of the quadratic map of $\tilde{Q}_{\cal E}$. This is easily obtained by considering the unital quadratic $\Phi$-algebra $\tilde{Q}$ in which $Q$ embeds. This algebra has an almost-classical quotient $\tilde{Q}_{\cal E}$ which is easily seen to contain a copy of $Q_{\cal E}$ as a subalgebra, and which is unital (its unit element being the class $[1,\Phi]$ where $1:\Phi \rightarrow \tilde {Q}$ is the constant map $1(\alpha)=1$
for all $\alpha \in \Phi$, where $1 \in \tilde{Q}$ is the unit of $\tilde{Q}$).

Our aim now is to prove that for a quadratic algebra $Q$ over $\Phi$ which is a nonsingular $\Phi$-module, the identities of $Q$ transfer
to its almost-classical algebra of quotients.

More precisely we have:

\begin{teorema} For any quadratic algebra $Q$ over a semiprime ring $\Phi$ as before, which is a nonsingular $\Phi$-module,
any quadratic homogeneous polynomial $F(x_1,\dots,x_n) \in FQ[X]$ which vanishes on $Q$ also vanishes on $Q_{\cal E}$.
\end{teorema}
\begin{proof} Since $F$ is homogeneous by hypothesis, all its monomials have the same degree $[k_1,\dots,k_n]$, where $k_i$
is the degree in $x_i$ (this is defined inductively for the generators of the free $Q$-word algebra over $X$ \cite[3.1.5]{jac-struc}).

 Now, if $q_1,\dots,q_n \in Q_{\cal E}$ we choose a representative $(f_i,\frak {a}_i) \in q_i$ for any $i$. Set $p=F(q_1,\dots,q_n)$, then for any
$\alpha_i\in \frak{a}_i$ we have $F(f_1,\dots,f_n)(\alpha_{1}^{k_1}\cdots \alpha_{n}^{k_n})=F(f_{1}(\alpha_1),\dots,f_{n}(\alpha_n))=0$,
hence $p=[F(f_1,\dots,f_n),\frak {a}_{1}^{[k_1]}\cdots \frak {a}_{n}^{[k_n]}]=0$, and therefore $p=F(q_1,\dots,q_n)=0$ thus $F$ vanishes on $Q_{\cal E}$.

\end{proof}

As a particular case of this result we turn our attention to the variety of algebras we are interested in, namely, quadratic Jordan algebras.

As mentioned before (\ref{pre}(5)), if a quadratic Jordan algebra $J$ is nondegenerate and PI, its centroid $\Gamma=\Gamma(J)$
is a reduced commutative algebra, and $J$ is a nonsingular $\Gamma$-module. The maximal algebra of quotients os $J$ is in this case
the almost-classical localization $J_{{\cal E} (J)}$, and therefore the previous result apply to it, thus yielding

\begin{teorema} If $J$ is a nondegenerate PI-algebra, then its maximal algebra of quotients $Q_{max} (J)$ is PI. Moreover, each homogeneous
quadratic polynomial vanishing on $J$ also vanishes on $Q_{max}(J)$.\end{teorema}

\begin{proof} This is immediate from the previous discussion taking into account that a Jordan algebra is PI if and only if
it satisfies a multilinear polynomial identity hence a homogeneous polynomial (\ref{pre}(5)).\end{proof}

\section{Orthogonal completeness}

We now introduce the technical tools developed in \cite{bmi} (see also \cite{bmami}) that allow the reduction of our problem to the factors of a subdirect decomposition, to which we will devote the next section, and that inherit the satisfaction of the logical sentences called Horn formulas.

We follow the notation of \cite{bmami}, and in particular we write $ \sum_{\epsilon\in
D}^{\perp}x_{\epsilon} \epsilon$ for the unique element such that $x \epsilon=x_{\epsilon}$ for all $\epsilon \in D$ where $D$ is a dense subset of the set $B=B({\cal{C}}(J))$ of idempotents of ${\cal{C}}(J)$ (see \cite[3.1.5]{bmami}).

\begin{lema}\label{ocquot}Let $J$ be a nondegenerate PI Jordan algebra. Then $Q_{max}(J)$ is
orthogonaly complete over $\C(J)$ (and over $\Gamma(Q_{max}(J))$).\end{lema}

\begin{proof} Consider the boolean algebra $B = B(\C(J))$ of all idempotents of $\C(J)$. Let $D
\subseteq B$ be a dense set of orthogonal idempotents. Put $I= \sum_{\epsilon\in
D}\epsilon J_\E$. Then $I$ is a nonzero ideal of $J_\E$ and if $0 \ne \gamma \in
\Gamma(J_\E)$ has $\gamma I = 0$, then $\gamma J_\E \subseteq {\rm Ann}_{J_\E}(I)$. Since
$J_\E$ is tight over $J$, we have $L = {\rm Ann}_{J_\E}(I)\cap J \ne 0$, and since $J$ is
PI, there is a nonzero $z \in C_w(J)\cap L$. Thus $0 \ne \delta = U_z\in \Gamma \subseteq
\C(J)$ has $\delta I = 0$, hence $\delta \epsilon = 0$ for all $\epsilon \in D$, hence
$\delta = 0$ by the density of $D$. This contradiction implies that $\gamma I \ne 0$ for
all $\gamma \in \Gamma(J_\E)$, hence $D$ is dense in $\Gamma(J_\E)$.

Take now an element $x_\epsilon \in Q_{max}(J) = J_\E$ for every $\epsilon \in D$, and
consider the ideal $\frak a = \sum_{\epsilon \in U}\Gamma(J_\E)\epsilon$ of $\Gamma(J_\E)$.
Define $f: \frak a \rightarrow J_\E$ by $f(\sum_{\epsilon \in D}\alpha_\epsilon\epsilon) =
\sum_{\epsilon \in D}\alpha_\epsilon\epsilon x_\epsilon$. Then $f$ is a homomorphism of
$\Gamma(J_\E)$-modules, and it defines an element $x =[f,\frak a] \in Q_{max}(J_\E) = J_\E$
which satisfies $\epsilon x = f(\epsilon) = \epsilon x_\epsilon$ for all $\epsilon \in D$.
Therefore $x = \sum^{\perp}_{\epsilon \in D}\epsilon x_\epsilon$, and this proves that
$J_\E$ is orthogonally closed.\end{proof}

Let $J$ be a nondegenerate Jordan PI-algebra, and let $X \subseteq Q_{max}(J)$ be a
subset. We denote by $O(X)$ the orthogonal complection of $X$. (Note that since
$Q_{max}(J)$ is orthogonally complete by \ref{ocquot}, $O(N)$ is the
intersection of all orthogonally complete subsets of $Q_{max}(J)$ that contain
$N$.)

\begin{lema}\label{oocc}Let $J$ be a nondegenerate Jordan PI-algebra, denote $\C = \Gamma(J)$,
and let $N \subseteq Q_{max}(J)$ be a $\C$-submodule of $Q_{max}(J)$. Then $O(N)$ is the
set $N_{{\cal E}(\C)}$ of all $[f,\frak a]\in J_\E = Q_{max}(J)= Q_{max}(J)_{{\cal E}(\C)}$
(with $\frak a \in {\cal E}(\C)$) such that $f(\alpha)\in N$ for all
$\alpha \in \frak a$, and in particular, $O(\C J) = J_\E = Q_{max}(J)$.\end{lema}

\begin{proof} We have $O(N) = \{\sum^\perp_{\epsilon \in D}\epsilon x_\epsilon\mid D$ is a dense
orthogonal subset of $B= B(\C)$, and $\{x_\epsilon \mid \epsilon \in D\} \subseteq N\}$ by
\cite[3.1.14]{bmami}. As in \ref{ocquot} it follows that the set $N_{{\cal E}(\C)}$ is
orthogonally complete, hence it contains
$O(N)$. On the other hand, if $x = \sum^\perp_{\epsilon \in D}\epsilon x_\epsilon$ for some
dense orthogonal subset $D \subseteq B$, and $\{x_\epsilon \mid \epsilon \in D\} \subseteq
N\}$, then the ideal $\frak a = \C D = \sum_{\epsilon \in D}\C(J)\epsilon$ is
essential. Moreover, if $\alpha \in \frak a$, then $\alpha = \sum_{\epsilon \in
D}\alpha\epsilon$, where there is only a finite number of nonzero products $\alpha\epsilon$.
Then $\alpha x = \sum_{\epsilon \in D}\alpha\epsilon x = \sum_{\epsilon \in D}\alpha\epsilon
x_\epsilon$ has only a finite number of nonzero summands, and all of them lie in $N$ since
$x_\epsilon \in N$ for all $\epsilon \in D$. Thus $x$ is represented in $Q_{max}(J) =
Q_{max}(J)_{{\cal E}(\C)}
$ by the mapping $f_x: \frak a\rightarrow Q_{max}(J)$ given by $f_x(\alpha)=
\alpha x$.\end{proof}

Recall from \cite[3.1.1]{bmami} that if $M$ is a nonsingular $\C$-module for a von Neumann regular
selfinjective  commutative ring $\C$, and $T\subseteq N$, there exists a unique idempotent
$\epsilon_T \in \C$ such that ${\rm Ann}_\C(T) = \C (1-\epsilon_T)$, and that
$\epsilon_{\delta T} = \delta\epsilon_T$ for any $\delta \in B$. In particular this applies
to $Q_{max}(J)$ for a nondegenerate Jordan PI-algebra $J$ and $\C = \Gamma(Q_{max}(J)) =
\C(Q_{max}(J))$, (\ref{centr}).

\begin{lema}\label{annid}Let $J$ be a nondegenerate Jordan PI-algebra, denote $Q =
Q_{max}(J)$, and $\C = \Gamma(Q)$. Let $X\subseteq Q$, and $I$ be the ideal
$id_{Q(J)}(X)$ generated by $X$. Then ${\rm Ann}_\C(X) ={\rm Ann}_\C(I) =
(1-\epsilon_X)\C$, and ${\rm Ann}_Q(I)= (1-\epsilon_X)Q$.\end{lema}

\begin{proof}Since ${\rm Ann}_\C(X) = (1-\epsilon_X)C$, it is clear that $(1-\epsilon_X)C = {\rm
Ann}_\C(I)$. Consider now the ideal $L = I +{\rm Ann}_Q(I)$ and the mapping $f: L \rightarrow Q$
given by $f(y+z) = y$ for $y \in I$, $z\in {\rm Ann}_Q(I)$. Since $L$ is essential this defines an element $\epsilon$ of $\cal C$ with $\epsilon x = f(x)$ for all $x \in L$, and clearly $\epsilon^2 =
\epsilon$. Now we have $(1-\epsilon){\rm Ann}_Q(I) = {\rm Ann}_Q(I)$, hence ${\rm Ann}_Q(I)
\subseteq (1-\epsilon)Q$, and since $\epsilon I = I$, it is clear that $(1-\epsilon)Q
\subseteq {\rm Ann}_Q(I)$, hence $(1-\epsilon)Q = {\rm Ann}_Q(I)$. Moreover
$(1-\epsilon_X)I = 0$ implies $(1-\epsilon_X)\epsilon L= 0$, hence
$(1-\epsilon_X)\epsilon= 0$ by the essentiality of $L$. On the other hand, $1-\epsilon \in
{\rm Ann}_C(I) = (1-\epsilon_X)C$, hence $(1-\epsilon) = (1-\epsilon)(1-\epsilon_X)=
1-\epsilon_X$, and we get $\epsilon=\epsilon_X$.\end{proof}

\begin{lema}\label{epsU}Let $Q$ be as before and $I$ be an ideal of $Q$. Then $a\in {\rm Ann}_Q(I)$
if and only if $\epsilon_a I = 0$.\end{lema}

\begin{proof} Denote by $H$ the ideal of $Q$ generated by $a$. If $a \in {\rm Ann}_Q(I)$, then $I
\subseteq {\rm Ann}_Q(H) = (1-\epsilon_a)Q$ by \ref{annid}, hence $\epsilon_a I = 0$. On the
other hand, if $\epsilon_a I = 0$, then $U_{\epsilon_a a}I = 0$, hence $\epsilon_a a \in
{\rm Ann}_Q(I)$. Since $(1-\epsilon_a)a = 0\in {\rm Ann}_Q(I)$, we get $a \in {\rm
Ann}_Q(I)$.\end{proof}

\section{Peirce stalks}

We assume here that $J$ is a nondegenerate PI-algebra with $J = Q_{max}(J)$. Denote by
$\C$ the centroid $\Gamma(J) = \C(J)$, and set $B = B(\C)$, the boolean algebra of
idempotents of $\C$. Let $\frak m$ be a maximal ideal of the boolean algebra $B$, recall that
the Peirce stalk of $J$ at $\frak m$ is the algebra $J_{\frak m} =J/\frak mJ$, where $\frak m J$ is
the set of elements $\sum_{i=1}^n\epsilon_i x_i$ for $\epsilon_i \in \frak m$ and $x_i\in J$,
which is easily seen to be an ideal of $J$.

We denote by $\phi_{\frak m}:J\rightarrow J_{\frak m}$ the projection $\phi_{\frak m}(x) =
x+\frak mJ$. As in \cite[3.2.4]{bmami}, we have ${\rm Ker}\, \phi_{\frak m} = \{x \in J\mid \epsilon_x\in
\frak m\}$, and for any orthogonally complete subset  $X \subseteq J$ containing $0$, we have
$\phi_{\frak m}(X) = 0$ if and only if $\epsilon_X \in \frak m$.

\begin{lema}\label{spelemnt}Let $J$ be a nondegenerate Jordan algebra, let $a, b\in J$, and denote
the ideals generated by $a$ and $b$ by $A$ and $B$ respectively. If
$U_aU_JU_bJ= 0$, then $A\cap B = 0$.\end{lema}

\begin{proof} Let $P$ be a strongly prime ideal of $J$ (i.e. $J/P$ is strongly prime). Denoting
with bars the projections modulo $P$, we have $U_{\bar a}U_{\bar J}U_{\bar b}\bar J = 0$,
hence either $\bar a = 0$ or $\bar b=0$ by the elemental characterization of strong
primeness \cite{elemental1}. Thus, either $a \in P$, hence $A\subseteq P$, or $b \in P$, hence $B
\subseteq P$. Therefore $A\cap B \subseteq P$ for any strongly prime ideal $P$. Since the
intersection of those is the McCrimmon radical, which vanishes in the nondegenerate algebra
$J$, we get $A\cap B = 0$.\end{proof}

\begin{lema}\label{elemidemp}Let $J$ be as before, and let $a, b \in J$. If $U_aU_JU_bJ = 0$, then
$\epsilon_a\epsilon_b = 0$.\end{lema}

\begin{proof} Let $B$ be the ideal of $J$ generated by $b$. By \ref{spelemnt}, $a\in {\rm Ann}_J(B)$,
hence $\epsilon_aB = 0$ by \ref{epsU}, hence $\epsilon_a \in {\rm Ann}_\C(b) =
(1-\epsilon_b)\C$ by \ref{annid}. Then $\epsilon_a \epsilon_b \in (1-\epsilon_b)\C \epsilon_b=(\epsilon_b-\epsilon^2_{b})\C=0$, and therefore $\epsilon_a\epsilon_b= 0$.\end{proof}

\begin{lema}\label{spstalk}Let $J$, $\C$, $B= B(\C)$ and $\frak m$ be as before, then $J_{\frak m}$ is
a strongly prime Jordan algebra.\end{lema}

\begin{proof} Suppose that $U_{\phi_{\frak m}(a)}U_{J_{\frak m}}U_{\phi_{\frak m}(b)}J_{\frak m}= 0$,
then $X = U_aU_JU_bJ \subseteq \frak m$. It is easy to see that $X$ is orthogonally closed,
hence $\epsilon_X \in \frak m$. Set $\delta = 1-\epsilon_X$. Then $U_{\delta a}U_JU_bJ =
\delta U_aU_JU_bJ = 0$, hence by \ref{elemidemp}, $0 = \epsilon_{\delta a}\epsilon_b=
\delta\epsilon_a\epsilon_b$. Thus $\delta\epsilon_a\epsilon_b \in \frak m$, and since $\delta
\not\in \frak m$, we get $\epsilon_a\in \frak m$ or $\epsilon_b\in \frak m$, hence $\phi_{\frak
m}(a) = 0$ or $\phi_{\frak m}(b) = 0$. Thus $J_{\frak m}$ is strongly prime by the elemental
characterization of strong primeness [ACNL].\end{proof}

\begin{lema}\label{centrstalk}Let $J$, $\C$, $B= B(\C)$ and $\frak m$ be as before. Then $\phi_{\frak
m}(C_w(J))= C_w(J_{\frak m})$.\end{lema}

\begin{proof} This is essentially the same proof as the one of \cite[3.2.13]{bmami} using the formula ${\bm \Phi} = (\forall x)(\forall y)(\Vert U_xU_zy = U_yU_zx\Vert)\land(\Vert x\circ (z\circ
y) = ((x\circ z)\circ y\Vert)$.\end{proof}

\begin{lema}\label{simplestalks}Let $J$, $\C$, $B= B(\C)$ and $\frak m$ be as before, then $J_{\frak
m}$ is a simple Jordan algebra of finite capacity.\end{lema}

\begin{proof} Since $J_{\frak m}$ is strongly prime by \ref{spstalk}, and it is PI, it suffices to
show that $\Gamma(J_{\frak m})$ is a field. To prove it, note that $C_w(J_{\frak m}) \ne
0$, and using \ref{centrstalk}, take
$z\in C_w(J)$ with $0 \ne \bar z = \phi_{\frak m}(z) \in C_w(J_{\frak m})$. Since $U_z\in \C$
and $\C$ is von Neumann regular, there exists $\gamma \in \C$ with $U_z^2\gamma=U_z\gamma U_z= U_z$. Then
$\gamma$ induces an element $\bar \gamma \in \Gamma(J_{\frak m})$ given by ${\bar
\gamma}(\phi_{\frak m}(x)) = \phi_{\frak m}(\gamma x)$ (since $\gamma\frak mJ\subseteq \frak
mJ$), and we have $U_{\bar z}^2\bar\gamma = U_{\bar z}$. Since $J_{\frak m}$ does not have
$\Gamma( J_{\frak m})$ torsion, we get $U_{\bar z}\bar \gamma = 1$, and $U_{\bar z}$ is
invertible. then for any $0 \ne \delta \in \Gamma(J_{\frak m})$, we have $\delta (\delta
U_{\delta
\bar z}^{-1}U_{\bar z}) = 1$, hence $\delta$ is invertible and $\Gamma(J_{\frak m})$ is a
field.\end{proof}

\begin{teorema}Any nondegenerate PI Jordan algebra is strongly nonsingular.\end{teorema}

\begin{proof} Let $K$ be an essential ideal of the nondegenerate PI Jordan algebra $J$, and
suppose that $U_aK = 0$ for some $a\in J$. Take $Q= Q_{max}(J)$ and denote $\C = \Gamma(Q)$.
Then the $\C$-linear span $\C J$ of $J$ in $Q$ is an algebra of quotients of $J$ by the
transitivity of algebras of quotients, and therefore any nonzero inner ideal of $\C J$ hits
$J$. This implies that the inner ideal $\C K$ of $\C J$ generated by $K$ over $\C$ is
essential, and it is clear that $U_a\C K = 0$. Now consider the orthogonal complection
$O(\C K)$, which by \ref{oocc} is easily seen to be an inner ideal of $Q$, and has $U_aO(\C
K)= 0$. Now, since $Q$ is an algebra of quotients of $J$, any nonzero inner ideal of $Q$
hits $J$ in a nonzero inner ideal, hence it hits $K$, and therefore it also hits $O(\C K)
\supseteq K$. Therefore $O(\C K)$ is essential. Now consider the sentence
$${\boldsymbol \Phi} = (\forall x)(\exists y)(\Vert x = 0\Vert \lor \Vert U_xy \ne
0\Vert)\land(\Vert U_xy \in O(\C K)\Vert).$$

It is easy to see that ${\boPHI}$ and $\neg{\boPHI}$ are Horn formulas, and  ${\bm
\Phi}$ is hereditary since $Q\models {\boPHI}$. Thus for any maximal ideal $\frak m$ of
the boolean algebra $B = B(\Gamma(Q))$, we obtain $Q_{\frak m}\models {\boPHI}$ by
\cite[3.2.11]{bmami}, hence $\phi_{\frak m}(O(\C K))$ is an essential inner ideal of $Q_{\frak m}$.
Therefore $\phi_{\frak m}(O(\C K)) = Q_{\frak m}$ by \ref{simplestalks}. Now $U_{\phi_{\frak
m}(a)}\phi_{\frak m}(O(\C K))= 0$ implies $\phi_{\frak m}(a) = 0$. Since this holds for all
$\frak m$, we get $a = 0$.\end{proof}

\begin{lema}\label{VNR}Let $J$ be a nondegenerate Jordan PI-algebra. Then $Q_{max}(J)$ is von
Neumann regular.\end{lema}

\begin{proof}Denote $Q = Q_{max}(J)$ and take $a \in Q$. For every maximal ideal $\frak m$ of
$B(\Gamma(J))$, the Peirce stalk $Q_{\frak m}$ is simple of finite capacity by
\ref{simplestalks}, hence there exists $y_{\frak m} \in J$ such that $\phi_{\frak
m}(U_ay_{\frak m}) = U_{\phi_{\frak m}(a)}\phi_{\frak m}(y_{\frak m}) = \phi_{\frak m}(a)$, Then
$b_{\frak m} = U_ay_{\frak m}-a \in \frak mJ$, hence $\epsilon_{b_{\frak m}} \in \frak m$. Set
$\delta_{\frak m} = 1-\epsilon_{b_{\frak m}}$, so that $\delta_{\frak m} \not\in \frak m$, and
$\delta_{\frak m}(U_ay_{\frak m} - a) = 0$. By \cite[3.1.2]{bmami}, there exists a finite collection
of idempotents $\epsilon_1, \dots, \epsilon_n$ with $\epsilon_i \le \delta_{{\frak m}_i}$
for some maximal ideal ${\frak m}_i$ of $B(\Gamma(Q))$, with $1 =
\epsilon_1+\dots+\epsilon_n$. Then $y = \epsilon_1y_{{\frak m}_1}+\dots + \epsilon_ny_{{\frak
m}_n}$ satisfies $U_ay = a$, and $a$ is von Neumann regular.\end{proof}


\end{document}